\documentclass[article,11pt]{amsart}
\usepackage[english]{babel}
\usepackage{latexsym,layout,verbatim,amssymb,amsmath,graphicx,setspace,cite}
\usepackage{amsthm,mathrsfs}
\usepackage{cleveref,enumerate,fontenc}
\usepackage[utf8]{inputenc}
\usepackage[T1]{fontenc}
\newenvironment{dedication}
  {
   \itshape             
   \raggedleft          
  }
  {\par 
  }
\newcommand{\spaces}{$\left(\Omega, \mathcal{A},\mathbb{P},\mathcal{F}\right)$}
\newtheorem{teor}{Theorem}[section]

\newtheorem{lemma}[teor]{Lemma}
\newtheorem{proposiz}[teor]{Proposition}

\newtheorem{defn}[teor]{Definition}

\newtheorem{osserv}[teor]{Remark}
\newtheorem{ass}{Assumptions}

\begin{document}

\title[A vector  Girsanov result ...]{A vector Girsanov result and its applications to conditional measures
 via the Birkhoff integrability. }
\thanks{-The Fondo Ricerca di Base 2018 
University of Perugia - 
  and  the GNAMPA -- INDAM (Italy) Project  "Dinamiche non autonome, analisi reale e applicazioni" (2018) supported this research.}

\begin{dedication}
In memory of Domenico Candeloro who is for us,\\ Master, Mentor and Friend, \\ \dag \  May $3^{\circ}$ 2019 \\ \mbox{~}

\end{dedication}
\author[Candeloro]{\bf Domenico Candeloro}
\address{Department of Mathematics and Computer Sciences \\ University of Perugia\\
Via Vanvitelli, 1 - 06123 Perugia (Italy), Orcid Id:  0000-0003-0526-5334}
\email{candelor@dmi.unipg.it}

\author[Sambucini]{\bf Anna  Rita Sambucini}
\address{\rm Department of Mathematics and Computer Sciences \\ University of Perugia\\ Via Vanvitelli, 1 - 06123 Perugia (Italy) , Orcid Id: 0000-0003-0161-8729} \email{anna.sambucini@unipg.it}

\author[Trastulli]{\bf Luca Trastulli}
\address{\rm Department of Mathematics and Computer Sciences \\ University of Perugia\\ Via Vanvitelli, 1 - 06123 Perugia (Italy) , Orcid Id:  0000-0002-7722-4008} 
\email{luca.trastulli@gmail.com}

\date{}

\begin{abstract} 
Some integration techniques for  real-valued  functions with 
respect to vector measures with values in  Banach spaces
 (and viceversa) are investigated in order to 
establish abstract versions of classical theorems of Probability and 
Stochastic Processes.
In particular the Girsanov Theorem is extended and  used 
with the treated
methods.
\end{abstract}
\maketitle

\section{Introduction}\label{one}

The theory of stochastic processes plays a very important role 
in  modelling various phenomena, 
 in a large class of disciplines such as physics, economics, 
statistics, finance, biology and chemistry.
 Then it is crucial to extend as much as possible the tools and 
results regarding this theory, in order to 
make them available even in abstract and general contexts.
A very important tool in  Measure Theory 
is the Girsanov  Theorem, strictly linked 
to the well-known Wiener stochastic process called the standard Brownian motion $(w_t)_{t\in[0,\infty)}$,
 defined on a probability space $\left(\Omega, \mathcal{A}, \mathbb{P}\right)$.
 A classic formulation of this result in the real case can be found in \cite{Mikosch} and it  
allows to change the underlying probability measure $\mathbb{P}$, through the definition of
Radon-Nikod\'ym derivative, in order to obtain an equivalent measure $Q$.
This turns out to be useful,
for example, in mathematical finance when  a neutral risk measure must be determined, in 
the Black-Scholes model.

Here we  generalize the Girsanov Theorem to the case of vector measure spaces following  the idea formulated in \cite{Mikosch} for the real Brownian motion and using 
the   Birkhoff vector integral studied in 
\cite{Fremlin0,Fremlin1,Fremlin2,cmn,cg,cg2016,cc, rodriguez1,rodriguez2,rodriguez3,cascales,potyrala,potyrala2,
fernandez,vm}, in \cite{bf}
for non additive-measures  and  in  \cite{Boccutosambucini, candeloro1,candeloro2,cascales1,ccgs2015a,cgsub1, dms16,cs2014,cs2015,bcs2014} for  the 
 multivalued integration. Other results on the Brownian motion subject are given also in \cite{lm10, GL,m09}.
This paper is inspired by \cite{girsanov,trastulli},
 in particular some of the results  were announced at the ICSSA 2018 Conference. \\
Now, we give a plan of the paper. In Section  \ref{two},
after an introduction of the Birkhoff integrals (Definitions \ref{b1} and \ref{birkhoff 2 specie}) the properties of such integrals are studied togheter with a link between them (Theorem \ref{3.3}). Moreover,  the notions of conditional expectation in this framework is given togheter with a tower property for martingales  and  Theorem \ref{teorema4}.
In section \ref{tre}  the main result: a vector version of the Girsanov result (Theorem \ref{GB}) is presented after having introduced the equivalent martingale measure and under Assumptions \ref{da1a3} and \ref{quarta}. At the end of this section an example is given satisfying Theorem \ref{GB}.
  Section \ref{quattro} is devoted to  applications of  Theorem \ref{GB} such as conditional measures (Proposition \ref{4.1} for $C([0,T])$-valued measures)   and to extensions of the  Itô representation of stochastic $X$-valued processes, using vector-stochastic integral and the classical Itô formulas.
In particular, in Theorem \ref{agB},   a process of the type
$$C_t={\scriptstyle (Bi_1)\hskip-.1cm}\int_0^t\Psi(s)ds+{\scriptstyle (Bi^{\star}_1)\hskip-.1cm}\int_0^t \Phi(s) dw_s, \quad \quad (\text{under}\,\,  \mathbb{P}),  \quad t\in [0,T]$$ is considered
and it is proved that it is possible 
 to eliminate the drift term in order to  obtain a local martingale.

\section{The Birkhoff integrals and their properties}\label{two}
Let $\left(\Omega,\mathcal{A},\nu\right)$ denote a measure space and $\left(X,\left\|\cdot \right\|\right)$ 
a Banach space.
From now on, with the letters $\mu, \nu$ we refer to scalar measures, that is
$\nu:\mathcal{A}\rightarrow \mathbb{R}^+_0$
 while we use letters as $N,M$ to denote vector measures, that is
 $N:\mathcal{A}\rightarrow X$.
 We also use capital letters like $X,Y$ to denote arbitrary Banach spaces, $X^{\ast},Y^{\ast}$ to denote
 their dual spaces and $x^{\ast},y^{\ast}$ the elements of the dual spaces. With letters like $\phi,\psi$
 we refer to scalar functions, and with $\Phi,\Psi$ to vector valued functions, defined on the measure space 
$\left(\Omega, \mathcal{A}, \nu\right)$.
We also denote with $\mathcal{B}(\mathbb{R}), (\mathcal{B}(I))$ as usual the Borel $\sigma$-algebra   on the real line
 (on the interval $I=[0,T], T>0$) and with $\mathcal{B}(X)$  the  Borel $\sigma$-algebra
 on $X$.
Finally, we shall denote by $a_t,b_t,z_t,$ scalar-valued stochastic processes, while $A_t,B_t,Z_t$ denote $X$-valued stochastic processes.
We recall that an $X$-valued stochastic process is a strongly measurable function
$Z:(I\times \Omega,\mathcal{B}(I)\otimes \mathcal{A})\rightarrow (X,\mathcal{B}(X))$.
There are many different versions of the Birkhoff integral (see \cite{birkhoff}). We shall use the
 following one, 
that provides two different kinds of integrals.
In order to do this we recall some basic definitions.
\bigskip
\begin{defn}\label{birkoff 1 specie} \rm
We say that a finite or countable family of non-empty  measurable sets  $\mathit{P}:=\left(U_j\right)_{j\in J}$ is a \emph{partition of } $\Omega$ if $U_j$ and $U_k$ are pairwise disjoint 
and they
 cover $\Omega$, that is $\bigcup_{j\in J}U_j=\Omega$.
We denote by
$\mathcal{P}$ the set of all partitions of the set $\Omega$.
Given two different partitions $\mathit{P},\mathit{P'}$ we say that the partition $\mathit{P}$ is \emph{finer than} 
$\mathit{P'}$ if for every $U\in \mathit{P}$ there exists a $U'\in \mathit{P'}$ such that $U\subset U'$.
\end{defn}
We distinguish between two different kinds of Birkhoff integral. The first one is about the integration
 of an $X$-valued function with respect to a scalar measure. The second one is the notion of integral 
of a scalar function with respect to a vector measure.
\begin{defn}[First type Birkhoff integral]\label{b1} \rm
Let $\Phi:\Omega \rightarrow X$ be a 
vector space-valued
 function and    $\nu:\mathcal{A}\rightarrow \mathbb{R}^+_0$ be
 a scalar, countably additive measure. Then
$\Phi$ is said to be \emph{first type Birkhoff integrable with respect to $\nu$}, briefly
 $\Phi\in Bi_1\left(\Omega,\nu\right)$ (or simply $Bi_1$ if there is not ambiguity about the space and the measure), 
if there exists $\mathcal{I} \in X$ \ such that for all $\varepsilon>0$ there exists a partition of $\Omega$,  
$P_{\varepsilon}$, such that, for every countable partition $\left(U_n\right)_{n\in \mathbb{N}}$ 
of $\Omega$, finer than $P_{\varepsilon}$ and for all  $\omega_n \in U_n$, it is
$\limsup_n \left\| \left(\sum_{k=1}^{n} \Phi(\omega_k)\nu(U_k)-\mathcal{I}\right)\right\|<\varepsilon$.
We call $\mathcal{I}\in X$ the 
\textit{Birkhoff integral of $\Phi$ with respect to $\nu$,} and we denote it by 
${\scriptstyle (Bi_1)\hskip-.1cm}\displaystyle{\int}_\Omega \Phi d\nu.$
\end{defn}
Now we are going
to define the second type of Birkhoff integral.
\begin{defn}
\label{birkhoff 2 specie} \rm
Let $\phi:\Omega \rightarrow \mathbb{R}$  and  $N:\mathcal{A}\rightarrow  X$ be a  
countably additive measure. Then
$\phi$ is said to be  \emph{second type Birkhoff integrable with respect to N},
 briefly $\phi\in Bi_2\left(\Omega,N\right)$ (or simply $Bi_2$),
 if there exists $\mathcal{I} \in X$ \ such that for all $\varepsilon>0$ there exists a partition of $\Omega$, 
 $P_{\varepsilon}$, such that, for every countable partition $\left(U_n\right)_{n\in \mathbb{N}}$ of 
$\Omega$, finer than $P_{\varepsilon}$ and for all  $\omega_n \in U_n$, it is
$\limsup_n \left\| \left(\sum_{k=1}^{n} \phi(\omega_k)N(U_k)-\mathcal{I}\right)\right\|<\varepsilon$.
We call $\mathcal{I}\in X$ the Birkhoff integral of $\phi$ with respect to $N$ and we denote it by 
${\scriptstyle (Bi_2)\hskip-.1cm}\displaystyle{\int}_\Omega \phi dN.$
\end{defn}
 The first (second) type Birkhoff integrability/integral of a function on a set $A \in \mathcal{A}$ is defined in the usual  manner 
since, thanks to a Cauchy criterion the integrability of the function restricted to $A$ follows immediately.

\begin{osserv}\label{pettisobtained}\rm
If $\mu$ is   $\sigma$-finite  then
the $Bi_1$ integrability is correspondent to  the classic Birkhoff integrability for Banach space-valued mappings (see also  
\cite[Theorem 3.18]{ccgs2015a}).
Moreover, 
  since the Birkhoff integral is stronger than the Pettis integral, it is clear that, as soon as $F$ is first type
 Birkhoff integrable with respect to $m$, the mapping $M:=A\mapsto {\scriptstyle (Bi_1)\hskip-.1cm}\displaystyle{\int}_A F dm$ is a countably additive measure.
\\
For the second type, 
 the mapping $M:=A\mapsto {\scriptstyle (Bi_2)\hskip-.1cm}\displaystyle{\int}_A f dN$ is weakly countably additive, since for
each $x^*$ in the dual space
$X^*$ the scalar mapping $f$ is integrable with respect to the scalar measure $x^*(N)$ 
(and to its variations).
Then, thanks to the Orlicz-Pettis Theorem (see \cite{pettis}), $M$ turns out to be also strongly countably additive.
\end{osserv}

We want to show a link between these two types of Birkhoff integral. Firstly we recall a result concerning the
 first type Birkhoff integrability.

\begin{teor}{\rm\cite[Th 3.14]{candeloro2}}
\label{3.1}
 Let $\Phi$ be a  strongly measurable and $Bi_1\left(\Omega,\nu\right)$-integrable function. Then, for every $\varepsilon>0$ there exists a countable partition 
$P^*:=\left\lbrace U_n, \ n\in \mathbb{N}\right\rbrace$ of measurable subsets of $\Omega$, such that,
 for every finer partition $P^{'}:=\left\lbrace V_m,\ m\in \mathbb{N}\right\rbrace$ of $P^*$ 
and for every $\omega_m\in V_m$, we have 
$$\sum_{m}\left\|\Phi(\omega_m)\nu(V_m) - {\scriptstyle (Bi_1)\hskip-.1cm} \int_{V_m} \Phi d \nu\right\|\leq \varepsilon.$$ 
\end{teor}

\begin{lemma}\label{remark sum} 
If $\phi$ is a scalar  measurable function and $\Phi$ is an $X$-valued strongly measurable function in $Bi_1\left(\Omega,\nu\right)$ then, given
$$(G_n)_n:= (\left\lbrace \omega \in \Omega:  \, n-1\leq|\phi(\omega)|\leq n\right\rbrace)_n \in \mathcal{A},$$
 for every $\varepsilon>0$  there exists a measurable
 countable partition $\left\lbrace U_n^j, \ j\in\mathbb{N}\right\rbrace$ of $G_n$ such that
for every finer partition 
$\left\lbrace V_n^j, \ j\in \mathbb{N}\right\rbrace$ and for every $\omega_n^j\in V_n^j$,
\begin{equation}\label{doppia somma}
\sum_{n}\sum_j\left\|\Phi(\omega_n^j)\phi(\omega_n^j)\nu(V_n^j)-\phi(\omega_n^j)N(V_n^j)\right\|
\leq2\varepsilon.
\end{equation}
\end{lemma}
\begin{proof}
By hypothesis 
 the product function $\phi(\omega)\Phi(\omega)$ is strongly measurable. 
Then, thanks to Theorem \ref{3.1}, for every $\varepsilon>0$ and for every $n$ there exists a measurable
 countable partition $\left\lbrace U_n^j, \ j\in\mathbb{N}\right\rbrace$ of $G_n$ such that, for every finer partition 
$\left\lbrace V_n^j, \ j\in \mathbb{N}\right\rbrace$ and for every $\omega_n^j\in V_n^j$, we obtain
$$\sum_{j}\left\|\Phi(\omega_n^j)\nu(V_n^j)-\int_{V_n^j}\Phi d\nu
\right\|\leq \dfrac{\varepsilon}{n2^n}.$$
Then it follows that 
\begin{equation*}
\sum_{n}\sum_j\left\|\Phi(\omega_n^j)\phi(\omega_n^j)\nu(V_n^j)-\phi(\omega_n^j)N(V_n^j)\right\|
\leq2\varepsilon.
\end{equation*}
\end{proof}

\begin{teor}\label{3.3}
Let $\phi:\Omega \rightarrow \mathbb{R}$ be a  measurable function and $\Phi  \in Bi_1(\Omega,\nu)$ be a vector valued strongly measurable function. We denote by 
$N(A)= {\scriptstyle (Bi_1)\hskip-.1cm}\displaystyle{\int}_{A}\Phi d\nu.$
Then 
 $\phi(\cdot)\Phi(\cdot)\in Bi_1\left(\Omega,\nu\right) \Longleftrightarrow 
\phi(\cdot)\in Bi_2\left(\Omega,N\right)$
and
\begin{equation}\label{equality}
{\scriptstyle (Bi_1)\hskip-.1cm}\int_{\Omega}\phi(\omega)\Phi(\omega)d\nu=
 {\scriptstyle (Bi_2)\hskip-.1cm}\int_{\Omega}\phi(\omega)dN.
\end{equation}
\end{teor}

\begin{proof}
Suppose that  $\phi(\omega)\Phi(\omega)\in B_1\left(\Omega,\nu\right)$ and let 
$J={\scriptstyle (Bi_1)\hskip-.1cm}\displaystyle{\int}_{\Omega} \phi(\omega)\Phi(\omega)d\nu.$
Then, fixed arbitrarily $\varepsilon>0$, we can find a  measurable partition $P^*:=\{U_n:n\in \mathbb{N}\}$ of $\Omega$ such that
$$\sum_n
\left\|\Phi(\omega_n)\phi(\omega_n)\nu(U'_n)- {\scriptstyle (Bi_1)\hskip-.1cm} \int_{U'_n}\Phi(\omega)\phi(\omega)d\nu \right\|\leq \varepsilon$$
for every finer partition $\{U'_n,n\in \mathbb{N}\}$ and  for every  $\omega_n\in U'_n$. 
So, taking a partition $\{V_m,m\in \mathbb{N}\}$ finer than
  $P^*$ and a partition 
$\{V_m^k, m,k\in \mathbb{N}\}$ 
given by Lemma \ref{remark sum}, using  (\ref{doppia somma}) we infer that
\begin{eqnarray*}
&& \hskip-1cm \left\|\sum_m\phi(\omega_m)N(V_m)-J\right\| = 
\left\| \sum_m\phi(\omega_m)N(V_m)-\sum_m {\scriptstyle (Bi_1)\hskip-.1cm} \int_{V_m}\Phi(\omega)\phi(\omega)d\nu\right\|\leq\\
& \leq&  \sum_m\left\|\phi(\omega_m)N(V_m)- {\scriptstyle (Bi_1)\hskip-.1cm} \int_{V_m}\phi(\omega)\Phi(\omega)d\nu\right\|\leq\\ &\leq& 
\sum_m\left\|\phi(\omega_m)N(V_m)-\phi(\omega_m)\Phi(\omega_m)\nu(V_m)\right\|+\\
&+&  \sum_m\left\|\phi(\omega_m)\Phi(\omega_m)\nu(V_m)- {\scriptstyle (Bi_1)\hskip-.1cm}\int_{V_m}\phi(\omega)\Phi(\omega)
d\nu\right\|\leq 3\varepsilon
\end{eqnarray*}
for every choice of $\omega_m \in V_m$.
Then $\phi\in Bi_2(\Omega,N)$, as we wanted.
Conversely, 
we assume that $\phi\in Bi_2(\Omega,N)$. Then, for every $\varepsilon > 0 $
there exists a countable measurable partition 
$P^{\#} :=\{U_n:n\in \mathbb{N}\}$ such that 
$$\limsup_n \|\sum_{i=1}^n \phi(\omega_i)N(U_i)- (Bi_2)\int_{\Omega} \phi dN\|\leq \varepsilon.$$
So, considering $\{V_k, k\in \mathbb{N}\}$  a finer partition than $P^{\#}$ and $\{V_j^k, j,k\in \mathbb{N}\}$,
 as in Lemma \ref{remark sum}, we get, using (\ref{doppia somma}),
\begin{eqnarray*}
&& \limsup_n\left\|\sum_{k=1}^n \phi(\omega_k)\Phi(\omega_k)\nu(V_k)-{\scriptstyle (Bi_2)\hskip-.1cm}\int_\Omega \phi dN\right\| \leq 
\\&\leq& 
\limsup_n\left\|\sum_{k=1}^n\phi(\omega_k)\Phi(\omega_k)\nu(V_k)-\sum_{k=1}^{n}
\phi(\omega_k){\scriptstyle (Bi_1)\hskip-.1cm}\int_{V_k}\Phi d\nu\right\|+\\
&+&\limsup_n\left\|\sum_{k=1}^n\phi(\omega_k)N(V_k)- {\scriptstyle (Bi_2)\hskip-.1cm}\int_{\Omega}\phi dN\right\|\leq 3\varepsilon.
\end{eqnarray*}
This shows that $\phi(\cdot)\Phi(\cdot)\in Bi_1(\Omega,\nu)$ and the equality (\ref{equality}) is true.
\end{proof}

In order to find applications in  stochastic processes, we need to extend notions like distribution 
of a function with respect to a vector measure.
\begin{defn} \rm
Let $\phi:\Omega\rightarrow \mathbb{R}$ be a  measurable  function and 
$N:\mathcal{A}\rightarrow X$ be
a countably additive measure.
We  define 
 the measure  induced by $\phi$ as
$N_\phi(B)=N(\phi^{-1}(B))$
 for every $B\in \mathcal{B}\left(\mathbb{R}\right).$ This measure  is countably additive and we call it the  
\emph{distribution of $\phi$ with respect to $N$}.
\end{defn}
We give now an integration by substitution result  for the $(Bi_2)$ integral.
\begin{teor}\label{3.4}
 Given two measurable functions $\phi: \Omega \to \mathbb{R}$, 
 $\psi:\mathbb{R}\rightarrow \mathbb{R}$, then the following relation holds:  
${\scriptstyle (Bi_2)\hskip-.1cm}\displaystyle{\int}_{\Omega}\psi(\phi)dN={\scriptstyle (Bi_2)\hskip-.1cm}\displaystyle{\int}_{\mathbb{R}}\psi(x)dN_\phi$
under the assumption that the integrals involved  exist.
\end{teor}
\begin{proof}
Suppose that the two integrals above exist as second type Birkhoff integrals. Now if we fix $x^{\ast}\in X^{\ast}$,
we can consider the real measures $x^{\ast}(N)$ and $x^{\ast}(N_{\phi})=(x^{\ast}(N))_{\phi}$. 
Then $\psi(\phi)$ is integrable with respect to $x^{\ast}(N)$ and $\psi$ is integrable with respect to $(x^{\ast}(N))_{\phi}$.
 So we obtain that
\begin{eqnarray*}
x^{\ast}\left({\scriptstyle (Bi_2)\hskip-.1cm}\int_{\Omega}\psi(\phi)dN\right)&=&\int_{\Omega}\psi(\phi)dx^{\ast}(N)=
\int_{\mathbb{R}}\psi(\omega)dx^{\ast}(N)_{\phi} = 
\\&=&
x^{\ast}\left({\scriptstyle (Bi_2)\hskip-.1cm}\int_{\mathbb{R}}\psi(\omega)dN_{\phi}\right).
\end{eqnarray*}
By the arbitrariness of $x^{\ast}$, the assertion follows.
\end{proof}
In order to introduce in this setting the definition of a martingale the notions of 
\emph{conditional expectation} and  \emph{filtration} are needed.
\begin{defn}\rm 
Let $\phi:\Omega\rightarrow \mathbb{R}$ be a  scalar function and $Bi_2(\Omega, N)$ integrable. We denote with $\sigma_\phi$ the sub-$\sigma$-algebra of $\mathcal{A}$, obtained by 
 taking all pre-images  $\phi^{-1}(B)$, for all $B\in \mathcal{B}(\mathbb{R})$.
In the special case of $\phi=z_t$, with  $t\in [0,T]$ fixed, then 
we define the \emph{natural filtration of $z_t$}, 
denoted by $\mathcal{F}_{z}=\left(\mathcal{F}_t\right)_t$, 
since 
$\mathcal{F}_t=\sigma_{z_t}$ for every $t$.
\end{defn}
Now we give
the notion of conditional expectation.

\begin{osserv}\rm
Given a sub $\sigma$-algebra $\mathcal F$
of $\mathcal{A}$, we say that an $X$-valued function is \textit{
$\mathcal{F}$-measurable}
if it is strongly measurable as 
a function from the measurable space
 $\left(\Omega,\mathcal{F}\right)$ to the measurable Banach space $\left(X,\mathcal{B}(X)\right)$.
\end{osserv}

\begin{defn}\rm
Let $\Phi\in Bi_1\left(\Omega,\nu\right)$ and $\mathcal{F}$ be a sub $\sigma$-algebra of $\mathcal{A}$. We define, provided that it exists, the 
\emph{conditional expectation of $\Phi$ with respect to $\mathcal{F}$}, indicated by $\mathbb{E}^{\nu}\left(\Phi|\mathcal{F}\right)$ ($\mathbb{E} \left(\Phi|\mathcal{F}\right)$ if there is no ambiguity about the measure), as the strongly 
 $\mathcal{F}$-measurable function $\Psi$ such that 
 $\Psi\in Bi_1(\Omega,\nu)$ and for every $E\in \mathcal{F}$ it holds 
$${\scriptstyle (Bi_1)\hskip-.1cm}\int_E \Phi d\nu={\scriptstyle (Bi_1)\hskip-.1cm}\int_E\Psi d\nu.$$
\end{defn}
From this the classic \emph{tower property} follows, i.e. for every sub $\sigma$-algebras of $\mathcal{A}$, 
such that $\mathcal{F}\subset \mathcal{G}\subset \mathcal{A}$, we have that
\begin{equation}\label{tower property}
\mathbb{E}(\Phi | \mathcal{F})=\mathbb{E}(\mathbb{E}(\Phi | \mathcal{G})|\mathcal{F}).
\end{equation}
Another important property of the conditional expectation that we can extend is the following.

\begin{teor}\label{teorema4}
Let $\Phi:\Omega\rightarrow X$ be a strongly measurable  function and $\mathcal{F}$ be a sub $\sigma$-algebra of $\mathcal{A}$, having
conditional expectation $\mathbb{E}\left(\Phi|\mathcal{F}\right)$. Then, given a  $\mathcal{F}$ measurable
 function $\phi:\Omega\rightarrow \mathbb{R}$ so that the product function $\Phi(\cdot)\phi(\cdot) \in Bi_1(\Omega, \nu)$,
 it holds:
$$\mathbb{E}\left(\Phi(\omega)\phi(\omega)|\mathcal{F}\right)=
\phi(\omega)\mathbb{E}\left(\Phi|\mathcal{F}\right).$$
\end{teor}
\begin{proof}
We claim that, for every $E\in \mathcal{F}$ 
the following relation is satisfied:
\begin{equation}\label{hahn}
{\scriptstyle (Bi_1)\hskip-.1cm}\int_E\Phi(\omega)\phi(\omega)d\nu={\scriptstyle (Bi_1)\hskip-.1cm}\int_E\phi(\omega)\mathbb{E}(\Phi(\omega)|\mathcal{F})
(\omega)d\nu.
\end{equation}
For every   $x^{\ast}\in X^{\ast}$ it holds 
$x^{\ast}\left(\mathbb{E}\left(\Phi|\mathcal{F}\right)\right)=
\mathbb{E}\left(x^{\ast}\left(\Phi\right)|\mathcal{F}\right)$, 
then we obtain that
\begin{eqnarray*}
x^{\ast}\left({\scriptstyle (Bi_1)\hskip-.1cm}\int_E \hskip-.1cm \Phi(\omega)\phi(\omega)d\nu\right) &=&
\int_E\hskip-.1cm x^{\ast}\left(\Phi(\omega)\right)\phi(\omega)d\nu=
\int_E \hskip-.1cm x^{\ast}\left(\mathbb{E}\left(\Phi|\mathcal{F}\right)\right)\phi(\omega)d\nu\\
&=&x^{\ast}\left({\scriptstyle (Bi_1)\hskip-.1cm}\int_E \hskip-.1cm \mathbb{E}\left(\Phi|\mathcal{F}\right)\phi(\omega)d\nu\right).
\end{eqnarray*}
So the function  $\omega\mapsto\mathbb{E}\left(\Phi|\mathcal{F}\right)\phi(\omega)$  is Pettis integrable with respect to $\mu$;
since the product is strongly 
measurable, 
this means that the product is $Bi_1(\Omega, \mu)$-integrable. Moreover, 
by the Hahn-Banach Theorem, the formula (\ref{hahn}) follows.
\end{proof}
Theorems \ref{3.3}, \ref{3.4} and \ref{teorema4} of this section were announced in \cite[Theorems 2,3,4]{girsanov} respectively without any proof. 
\section{The Girsanov Theorem for vector measures}\label{tre}
 To the aim of extending 
 the Girsanov Theorem to the  Banach-valued measures  
we shall  find the distribution of a scalar valued stochastic process under a vector measure $N$, make a transformation of this process, compute its new distribution and then define a new measure using these two density functions.

First,
we need to define the concept of martingale when we use a 
vector measure $N$.
Let a scalar measure $\nu:\mathcal{A}\rightarrow \mathbb{R}^+_0$ be fixed and $\Phi:\mathcal{A}\rightarrow X$
 be an $X$-valued function strongly measurable and $B_1(\Omega,\nu)$ integrable. We define the vector measure
 $N:\mathcal{A}\rightarrow X$ as follows:
$
N(A) :={\scriptstyle (Bi_1)\hskip-.1cm}\displaystyle{\int}_A \Phi d\nu.$
 A measure $Q$ is equivalent to a measure $N$ ($Q \sim N$) if there exists a $Bi_2(\Omega, N)$ integrable and positive function $\varphi$ such that $dQ/dN = \varphi$.
So  the definition of martingale for a scalar process
 $\left(z_t\right)_t$ can be given.
\begin{defn}\rm
Let $\left(z_t: \Omega \to \mathbb{R} \right)_{t \in [0,T]}$ be a scalar stochastic process on the probability space $\left(\Omega,\mathcal{A},N\right)$, 
where $N$ is as above. We say that $z_t$ is a \emph{N-martingale in itself}, that is a martingale with respect to
 its natural filtration $\mathcal{F}_{z}=\left(\mathcal{F}_t\right)_t$, 
if for every $s<t, s,t\in [0,T]$, we have that $\mathbb{E}^N\left(z_t|\mathcal{F}_s\right)=z_s$. 
\end{defn}
\begin{osserv}\rm
The identity $\mathbb{E}^N\left(z_t|\mathcal{F}_s\right)=z_s$
 in terms of integrals  means that, for every 
$E\in \mathcal{F}_s$,  it is 
${\scriptstyle (Bi_2)\hskip-.1cm}\displaystyle{\int}_E z_t dN={\scriptstyle (Bi_2)\hskip-.1cm}\displaystyle{\int}_E z_s dN.$
\end{osserv}

\begin{defn}\rm 
Let $\left(z_t: \Omega \to \mathbb{R} \right)_{t \in [0,T]}$ be a scalar stochastic process on the probability space \spaces \, which is adapted to $\mathcal{F}$.
A vector-valued measure $Q$ is called an \emph{equivalent martingale measure}  
for  $(z_t)_t$  with respect to $\mathbb{P}$ 
if $(z_t)_t$ is $Bi_2$-integrable,  it is  a martingale with respect to $Q$ and $Q \sim \mathbb{P}$.
\end{defn}
  In financial market the equivalent martingal measure is called also the risk-neutral measure.
Now we give some basic assumptions on stochastic processes. 
\begin{ass}\label{da1a3}\rm
 Let us  assume that $\left( z_t\right)_{t\in [0,T]}$ is a stochastic scalar process 
 defined as usual on the space $\left(\Omega, \mathcal{A},N\right)$  such that these
 conditions are satisfied:
\begin{enumerate}
\item[\ref{da1a3}.a)] \, The function $\omega\mapsto z(t,\omega)$ belongs to the space $Bi_2(\Omega,N)$, for  
every $t\in [0,T]$, with null integral and admits a density function, that means that its distribution under the vector
 measure $N$, denoted by  $N_t(B):=N_{z_t}(B)=N(z_t^{-1}(B))$, for every 
$\mathcal{B}(\mathbb{R})$
could be written as a Birkhoff first type integral of some vector function $F_t:\mathbb{R}\rightarrow X \in Bi_1(\Omega,\lambda)$   (where $\lambda$ is  the Lebesgue measure), namely:
$$N_t(B)={\scriptstyle (Bi_1)\hskip-.1cm}\int_B F_t(x)dx.$$
\item[\ref{da1a3}.b)]Let $\tilde{z}_t=z_t+\theta(t)$, where $\theta:[0,T]\rightarrow \mathbb{R}$ is a 
measurable function. Suppose that for every $t\in [0,T]$ there exists a scalar positive function $g_t(x)$ such that the following factorization holds 
$F_t(x)=g_t(x)\tilde{F}_t(x)=g_t(x)F_t(x-\theta(t)).$
\item[\ref{da1a3}.c)]The process defined as $y_t=\left(g_t(\tilde{z}_t)\right)_t$  is a $N$-martingale in itself.
\end{enumerate}
\end{ass}
\begin{osserv} \rm
Conditions \ref{da1a3}.a) and \ref{da1a3}.b) allow us to work with the distributions of $z_t$ and $\tilde{z}_t$ under the measure $N$. We note that by \ref{da1a3}.a) it follows that the process $\tilde{z}_t$ defines again a $Bi_2(\Omega,N)$ random variable, for every $t\in[0,T]$ and it admits a density function too, of
type
$\tilde{F}_t:=F_t(x-\theta(t)).$
In fact, by Lebesgue integral translation 
invariance, substituting $r=x+\theta(t)$, we obtain
$$\tilde{N}_t(B):=N_t(B+\theta(t))=
{\scriptstyle (Bi_1)\hskip-.1cm}\int_{B+\theta(t)}F_t(x)dx=
{\scriptstyle (Bi_1)\hskip-.1cm}\int_B F_t(r-\theta(t))dr.$$
So, thanks to Theorem 
\ref{3.3}, we have
$N_t(B)={\scriptstyle (Bi_1)\hskip-.1cm}\int_{B}g_t(x)dN_{\tilde{z}_t}.$
It is important to note that $g_t$ is a positive function, so, by Assumptions \ref{da1a3}.c)
 we know that $y_t=g_t(\tilde{z}_t)$
 is a $N$-martingale.
 \end{osserv}

Finally, we recall that,  when $\{z_t\}_t$ is the classical (scalar) Brownian Motion, then
it turns out that
$g_t(x)=exp(-qx+\frac{1}{2}q^2t)$
and so the process
\begin{eqnarray}\label{esempio}
\{g_t(\widetilde{z_t})\}_t
=\{exp(-qz_t-\frac{1}{2}q^2t)\}_t
\end{eqnarray}
is a martingale with respect to the natural filtration of $\{w_t\}_t$ (see e.g. 
\cite[(4.20)]{Mikosch}, \cite{novikov}).\\

So, we  define the change of measure setting the new vector measure
$Q:\mathcal{A}\rightarrow X$
 as follows:
\begin{equation}\label{measure q}
Q(A) = {\scriptstyle (Bi_2)\hskip-.1cm}\int_Ay_T dN={\scriptstyle (Bi_1)\hskip-.1cm}\int_Ay_T\Phi d\nu,
\end{equation}
for every  $A\in \mathcal{A}$. Observe that $Q \sim N$.
Under Assumptions \ref{da1a3}, 
the marginal distribution of the 
stochastic process
$z_t$ is preserved when we change the underlying measure. This means that the distribution of $z_t$ 
under $N$ is the
same of the process $\tilde{z}_t$ under the new measure $Q$.

\begin{teor}\label{teorema5}
Let $N, Q, z_t, \tilde{z}_t$ be the vector measures and the scalar stochastic processes defined 
before. Under Assumptions \ref{da1a3},  
the marginal distributions of these two processes are preserved 
under the change of measure, that is, for every $t\in [0,T]$, 
$N_{z_t}= Q_{\tilde{z}_t}$.
\end{teor}

\begin{proof}
We fix $B\in \mathcal{B}(\mathbb{R})$ and $t\in [0,T]$. By the Assumptions \ref{da1a3}.a) and \ref{da1a3}.b) we have that:
$Q_{\tilde{z}_t}(B)=Q \left((\tilde{z}_t)^{-1}(B\right))={\scriptstyle (Bi_2)\hskip-.1cm}\displaystyle{\int}_{(\tilde{z}_t)^{-1}(B)}y_TdN.$
Now, using the fact that $y_t$ is a martingale,
that is the Assumption \ref{da1a3}.c)  with respect to
$N$, and by Theorem \ref{3.4}, 
we write
\begin{eqnarray*}
\displaystyle {\scriptstyle (Bi_2)\hskip-.1cm}\int_{(\tilde{z}_t)^{-1}(B)}y_TdN 
&=&{\scriptstyle (Bi_2)\hskip-.1cm}\int_{(\tilde{z}_t)^{-1}(B)}y_tdN=
{\scriptstyle (Bi_2)\hskip-.1cm}\int_{(\tilde{z}_t)^{-1}(B)}g_t(\tilde{z}_t)dN=\\
&=&{\scriptstyle (Bi_2)\hskip-.1cm}\int_{B}g_t(r)dN_{\tilde{z}_t}.
\end{eqnarray*}
Since this holds for every $B\in \mathcal{B}(\mathbb{R})$, the proof is complete. 
\end{proof}

Our aim is to prove that 
$Q$ is an equivalent martingale measure for $(\widetilde{z}_t)_t$  with respect to $N$.
In order to do this, we assume
\begin{ass}\label{quarta}\rm
 The product process $\left(\tilde{z}_ty_t\right)_t$ is a martingale in itself with respect to $N$. 
\end{ass}

Then, we are ready to formulate the main theorem, that is
the Girsanov Theorem for vector measures 

\begin{teor}[{\rm Girsanov Theorem}]\label{GB}
Under 
Assumptions \ref{da1a3} and \ref{quarta} 
 the process $\left( \tilde{z}_t\right)_{t}$ is a $Q$-martingale in itself,  where $Q$ is
the vector measure, defined  in (\ref{measure q}).
\end{teor}
\begin{proof}
By assumptions \ref{da1a3}.a) and \ref{da1a3}.b), we can define the positive real process $y_t$ and by means of \ref{da1a3}.c),
 we have that $Q \sim N$. 
Then it is 
${\scriptstyle (Bi_2)\hskip-.1cm}\int_E\tilde{z}_t d Q={\scriptstyle (Bi_2)\hskip-.1cm}\int_E\tilde{z}_ty_T dN.$ 
Now we fix $s,t\in [0,T]$, with $s\leq t$ and $E\in \mathcal{F}_s$. Using 
the martingale property of the process
 $y_t$, we have
\begin{eqnarray*}
{\scriptstyle (Bi_2)\hskip-.1cm}\int_E\tilde{z}_t d Q&=&
{\scriptstyle (Bi_2)\hskip-.1cm}\int_E \hskip-.2cm \tilde{z}_ty_T dN= 
{\scriptstyle (Bi_2)\hskip-.1cm}\int_E \hskip-.2cm z_ty_t dN+
{\scriptstyle (Bi_2)\hskip-.1cm}\int_E \hskip-.2cm \theta(t)y_T dN= \\
&=&{\scriptstyle (Bi_2)\hskip-.1cm}\int_E  z_ty_tdN+{\scriptstyle (Bi_2)\hskip-.1cm}\int_E\theta(t)y_sdN= \\ &=&
{\scriptstyle (Bi_2)\hskip-.1cm}\int_Ez_ty_tdN+\theta(t){\scriptstyle (Bi_2)\hskip-.1cm}\int_Ey_sdN=\\
&=&{\scriptstyle (Bi_2)\hskip-.1cm}\int_Ez_ty_tdN+\theta(t)Q(E).
\end{eqnarray*}
Now, from Assumptions \ref{quarta} and \ref{da1a3}.c) we 
deduce
\begin{eqnarray*}
{\scriptstyle (Bi_2)\hskip-.1cm}\int_E \hskip-.2cm \tilde{z}_t dQ=
{\scriptstyle (Bi_2)\hskip-.1cm}\int_E \hskip-.2cm(z_sy_s+\theta(s)y_s) dN= 
{\scriptstyle (Bi_2)\hskip-.1cm}\int_E \hskip-.2cm z_sy_s dN+
 \theta(s) {\scriptstyle (Bi_2)\hskip-.1cm}\int_E \hskip-.2cm y_s dN
\end{eqnarray*}
but using the tower property (\ref{tower property}), we have 
$${\scriptstyle (Bi_2)\hskip-.1cm}\int_Ez_sy_T dN= {\scriptstyle (Bi_2)\hskip-.1cm}\int_Ez_sy_s dN,$$
and then we get
\begin{eqnarray*}
{\scriptstyle (Bi_2)\hskip-.1cm} \int_E\tilde{z}_t dQ &=&(Bi_2) \int_Ez_sy_s dN+\theta(s) (Bi_2)\int_Ey_s dN= \\ &=&
{\scriptstyle (Bi_2)\hskip-.1cm}\int_Ez_sy_T dN+\theta(s) {\scriptstyle (Bi_2)\hskip-.1cm}\int_Ey_T dN=\\
&=&{\scriptstyle (Bi_2)\hskip-.1cm} \int_E(z_s+\theta(s)) dQ={\scriptstyle (Bi_2)\hskip-.1cm} \int_E\tilde{z}_s dQ.
\end{eqnarray*}
The last relation is the martingale property of the process $\tilde{z}_t$ and the proof is completed.
\end{proof}
Theorems
\ref{teorema5} and \ref{GB} were announced in \cite[Theorems 5, 6]{girsanov};
 the last one with also a brief sketch of the proof.\\

If we consider, for example,
a Brownian motion $\left(w_t\right) _{t\geq 0}$,
 we know that this process is a martingale on
 this space with 
respect to this filtration $\mathcal{F}$. However, if we condition it on an expiration time $T>0$ fixed, the 
distribution of  this process changes and in general, it doesn't preserve some of its properties, such as the martingale property (see for example 
\cite[Theorem 5.4]{chang}).

Let 
\begin{equation}\label{condit measure 1}
N:=\mathbb{P}\left(\cdot|w_T\right):\mathcal{A}\rightarrow L^1 \left(\Omega\right).
\end{equation}
(i.e. $N(A)=\mathbb{P}(A|w_T)=E(1_A|w_T)$).
If we study the future time process $\left(w_{t+T}\right)_{t>0}$ under this new 
measure, we can observe that
 the stochastic process
 $\tilde{w}_t=w_{t+T}+qt$ is an $N$-martingale.
In fact, 
since $N(A)=\mathbb{P}\left(A|w_T\right)=\mathbb{E}\left(1_A|w_T\right),$
 the last term is a $L^1$ random variable, depending on $w_T$ and it shows that $N$ is a vector $L^1$ measure.
We observe that the process $w_{t+T}$ has the same distribution of $w_t+w_T$, where here we consider $w_t$ as independent of $w_T$, and, under the measure $N$, it has a Gaussian distribution of parameters $\mathcal{N}(w_T,t).$ Then, it admits the following density function:
$$f_t(x)=\dfrac{1}{\sqrt{2\pi t}}\exp\left(-\dfrac{(x-w_T)^2}{2t}\right).$$
Now, if we consider the transformed process $(\tilde{w_t})_t=(w_{t}+w_T+qt)_t$, its marginal distribution admits a conditional density given by
$$\tilde{f}_t(x)=\dfrac{1}{\sqrt{2\pi t}}\exp\left(\dfrac{-(x-qt-w_T)^2}{2t}\right).$$
Their quotient   is 
$$g_t(x)=\dfrac{f_t(x)}{\tilde{f}_t(x)}=\exp\left(\frac{q^2t}{2}-q(x-w_T)\right).$$
Then, evaluating  $g_t(\tilde{w}_{t+T})$, 
we obtain the process
$$y_t:=g_t(\tilde{w}_t)=\exp\left(-\frac{q^2t}{2}-qw_t\right),$$
again considering $w_t$, and then the process $y_t$, 
independent 
of $w_T$. 
Then the process $y_T$, that we know to be a martingale under $\mathbb{P}$ (by 
\cite[(4.20)]{Mikosch}),
 since it is independent by $w_T$ is a martingale even under the conditional probability $N=\mathbb{P}|w_T$.
Furthermore, we have that the process 
$$\widetilde{w}_tY_t=(w_t+qt)g_t(\widetilde{w}_t)+ w_Tg_t(\widetilde{w}_t)$$
is the sum of two martingales under the conditional probability $N$, again considering $w_t$ as independent by $w_T$,
 so it is a martingale too. Then the process $(w_{t+T})_{t>0}$, under the probability $N$ obtained conditioning with
 respect to $w_T$, satisfies all the condition of the Girsanov Theorem \ref{GB}
 and then the new process
 $\left(\tilde{w}_t\right)_t$ is a martingale under
$N=\mathbb{P}|w_T$.

\section{Some applications}\label{quattro}
Since conditional measures can be seen as vector measures, 
 using the tools obtained previously about the Birkhoff integral, we give some examples of applications of the 
Girsanov Theorem \ref{GB}. This could be useful for example  when conditioning of random variables to future  (or past times) is considered.
Now consider  a $C([0,T])$-valued stochastic process $\Phi$ defined as
\begin{eqnarray}\label{qui}
\Phi(\omega,t) &=&\exp\left\{-w_{\tau}-w_t+w_{t\wedge\tau}-\dfrac{t+\tau-\tau\wedge t}{2}\right\} \nonumber = \\ &=&
\begin{cases}
\exp\{-w_{\tau}-\tau/2\} \ \  \textit{if} \ \ t\leq \tau \\   
\exp\{-w_t-t/2\}\ \ \ \textit{if} \ \ t> \tau 	
\end{cases}
\end{eqnarray}
with $\tau \in [0,T]$  and
\begin{eqnarray}\label{Ncirc}
N^{\circ} \left(A\right)=\int_A\Phi (\omega, T) d\mathbb{P},
\end{eqnarray} be a $C([0,T])$-valued measure.
\begin{proposiz}\label{4.1}
Let $(w_t)_{t\in[0,T]}$ be a Brownian motion on the filtered probability space \spaces. 
Let $F_t$ and $\tilde{F}_t$ be the density functions  
of $w_t$ and $\tilde{w_t}=w_t+t$ under 
$N^{\circ}$, respectively.
Then, for every $t\in[0,T]$ there exists a real function $g_t(x)$ such that $F_t(x)=g_t(x)\tilde{F}_t(x)$ for every 
$x\in \mathbb{R}$ and the vector measure $Q$, 
defined by
$\dfrac{dQ}{dN^{\circ}}=g_t(\tilde{w}_t),$
is an equivalent martingale measure for the process $\tilde{w}_t$.
\end{proposiz}

\begin{proof}
We are going to prove that
for every $s\leq t$ in $[0,T]$
\begin{equation}
\mathbb{E}\left(\Phi (\omega,t)|\mathcal{F}_s\right)=\exp\{-w_s-\dfrac{1}{2}s\},
\end{equation}
namely it is independent of $\tau$.
We distinguish three cases:
\begin{description}
\item[($\tau\leq s\leq t$)]  in this case we have  $\Phi (\cdot, t) =\exp\{-w_t-\dfrac{1}{2}t\}$ so, being a martingale process, it
is trivial to deduce that 
$\mathbb{E}\left(\Phi (\omega,t)|\mathcal{F}_s\right)=\exp\{-w_s-\dfrac{1}{2}s\}.$
\item[($s\leq \tau \leq t$)]  again $\Phi(\omega,t)=\exp\{-w_t-\dfrac{1}{2}t\}$, then, as before, we get
$\mathbb{E}\left(\Phi (\omega, t)|\mathcal{F}_s\right)=\exp\{-w_s-\dfrac{1}{2}s\}.$
\item[($s\leq t\leq \tau$)] 
in this case $\Phi (\omega, t)=\exp\{-w_{\tau}-\dfrac{1}{2}\tau\}$, and since it is a 
martingale, it follows: 
$\mathbb{E}\left(\Phi(\omega,t)|\mathcal{F}_s\right)=\exp\{-w_{s}-\dfrac{1}{2}s\}.$
\end{description}
Now, considering the vector measure $N^{\circ}$, it is
$$\int_Ew_tdN^{\circ}=
\int_E w_t\Phi (\omega, T) d\mathbb{P}=
\int_E w_t\exp\{-w_t-\dfrac{1}{2}t\}d\mathbb{P},$$
for every $E\in \mathcal{F}_s$, where we have used the tower property (\ref{tower property}), that is
\begin{eqnarray*}
\mathbb{E}^{\mathbb{P}}\left(w_t\Phi(\omega, T)|\mathcal{F}_s\right) &=&\mathbb{E}^{\mathbb{P}}
\left(\mathbb{E}^{\mathbb{P}}\left(w_t\Phi(\omega, T)|\mathcal{F}_t\right)|\mathcal{F}_s\right)= \\ &=&
\mathbb{E}^{\mathbb{P}}\left(w_t
\mathbb{E}^{\mathbb{P}}\left(\Phi(\omega, T)|\mathcal{F}_t\right)|\mathcal{F}_s\right)=\\
&=& \mathbb{E}^{\mathbb{P}}\left(w_t\exp\{-w_t-\dfrac{1}{2}t\}|\mathcal{F}_s\right).
\end{eqnarray*}
This shows that $w_t\Phi(\omega, t)$ is not a martingale under $\mathbb{P}$ and this is equivalent to say that it is not a martingale under $N^{\circ}$.
Since the marginal distributions of the processes $w_t$ and $\tilde{w}_t$ under $\mathbb{P}$ are Gaussian with 
parameters respectively $\mathcal{N}\left(0,t\right)$ and $\mathcal{N}\left(t,t\right)$, the marginal distributions 
under $N^{\circ}$ admit a density function  obtained in this way:
\begin{eqnarray*}
N^{\circ}_t\left(B\right)&=&\int_{(w_t)^{-1}(B)}dN^{\circ}=\int_{(w_t)^{-1}(B)}\Phi(\omega, T) d\mathbb{P}= \\ &=& \int_B\mathbb{E}(\Phi(\omega, T)|\mathcal{F}_t)d\mathbb{P}_{w_t}=\\
&=&\int_B\exp\{-x-\dfrac{1}{2}t\}\dfrac{1}{\sqrt{2\pi t}}\exp\left\lbrace-\dfrac{x^2}{2t}\right\rbrace dx.
\end{eqnarray*}
Then we have the following densities function under $N^{\circ}$:
\begin{eqnarray*}
F_t(x)&=&\dfrac{1}{\sqrt{2\pi t}}\exp\left\lbrace-x-\dfrac{1}{2}t-\dfrac{x^2}{2t}\right\rbrace,\\
\tilde{F}_t(x)&=&\tilde{F}_t(x-t)=\dfrac{1}{\sqrt{2\pi t}}\exp\left\lbrace-x-t-\dfrac{1}{2}t-\dfrac{(x-t)^2}{2t}\right\rbrace,
\end{eqnarray*}
and then it is easy 
to see that the function $g_t$ we are looking for is given by
$$\dfrac{f_t(x)}{\tilde{f}_t(x)}=g_t(x)=\exp\left\lbrace-\dfrac{1}{2}t-x\right\rbrace.$$
Then we have that
$y_t=g_t(w_t+t)=\exp\left\lbrace-\dfrac{3}{2}t-w_t\right\rbrace$
and we want to prove that it is a martingale under the measure $N^{\circ}$, which is equivalent to prove that the vector process
 $\Phi(\omega, T)y_t$ is a $\mathbb{P}$-martingale. 
Again, using the tower property (\ref{tower property}), we 
find that
\begin{eqnarray*}
\mathbb{E}^{\mathbb{P}}\left(y_t\Phi(\omega, T)|\mathcal{F}_s\right)&=&\mathbb{E}^{\mathbb{P}}
\left(\mathbb{E}^{\mathbb{P}}\left(y_t\Phi(\omega, T)|\mathcal{F}_t\right)|\mathcal{F}_s\right)=\\ &=&
\mathbb{E}^{\mathbb{P}}\left(y_t
\mathbb{E}^{\mathbb{P}}\left(\Phi(\omega, T)|\mathcal{F}_t\right)|\mathcal{F}_s\right)=\\
&=&\mathbb{E}^{\mathbb{P}}\left(y_t\exp\{-w_t-\dfrac{1}{2}t\}|\mathcal{F}_s\right)=\\ &=&
\mathbb{E}^{\mathbb{P}}\left( \exp\left\lbrace-2w_t-2t\right\rbrace|\mathcal{F}_s\right)=\\
&=&\exp\{-2w_s-2s\}=\mathbb{E}^{\mathbb{P}}(y_s\Phi(\omega, T)|\mathcal{F}_s),
\end{eqnarray*}
because $\exp\{-2w_s-2s\}$ is a martingale under $\mathbb{P}$.
Since  $y_t$ is a positive $N^{\circ}$-martingale
satisfying Assumptions \ref{da1a3} with respect to $N^{\circ}$ and with $\theta(t) = t$,
 we define a new vector measure as
$\displaystyle{Q\left(A\right):=\int_Ay_T dN^{\circ}, \ A\in \mathcal{A}}$.
Now, recalling the increment invariance of the Brownian 
motion and the expected value of the
log normal random variable, we prove that
\begin{eqnarray*}
&&\mathbb{E}^{Q}\left(w_t+t|\mathcal{F}_s\right)=\mathbb{E}^{N^{\circ}}
\left(w_ty_t|\mathcal{F}_s\right)+t\mathbb{E}^{N^{\circ}}\left(y_t|\mathcal{F}_s\right)= \\ &=& 
\mathbb{E}^{\mathbb{P}}
\left(w_ty_t\Phi (\omega, T)|\mathcal{F}_s\right)+ty_s=
\mathbb{E}^{\mathbb{P}}\left(w_t\exp\{-2w_t-2t\}|\mathcal{F}_s\right)+ty_s=\\
&=&\mathbb{E}^{\mathbb{P}}\left(\left(w_t-w_s\right)\exp\{-2w_t+2w_s\}|\mathcal{F}_s\right)
\exp\left\lbrace-2w_s-2t\right\rbrace+ty_s+\\
&+&w_s\exp\left\lbrace-2w_s-2t\right\rbrace=\mathbb{E}^{\mathbb{P}}\left(\left(w_t-w_s\right)\exp\{-2w_t+2w_s\}\right)\exp\left\lbrace-2w_s-2t\right\rbrace+\\
&+&ty_s+w_sy_s=-(t-s)\exp\{-2s+2t\}\exp\{-2w_s-2t\}+(t+w_s)y_s=\\
&=&(s-t)y_s+(t+w_s)y_s=(w_s+s)y_s=\tilde{w}_sy_s.
\end{eqnarray*}
This means that the process $\tilde{w}_ty_t$ is an $N^{\circ}$-martingale and this proves Assumptions \ref{quarta} with respect $N^{\circ}$.
So, using the
Girsanov Theorem \ref{GB}, the process  
 $\tilde{w}_t$ is a $Q$-martingale and the proof is completed. 
\end{proof}

{\bf A Birkhoff integral representation}
 Using the classical Girsanov Theorem, it is possible to change the drift term of an Itô integral and to
obtain a local martingale. We want to prove an analogous result for vector processes that have an integral representation in terms of Birkhoff integrals (Theorem \ref{agB}).
The main problem is to adapt to this theory the expression
\begin{equation}\label{eqbirkhoff ito}
A_t={\scriptstyle (Bi_1)\hskip-.1cm}\int_0^t\Psi(s)ds+{(\cdot)}\int_0^t\Phi(s)dw_s, \quad (\text{under} \ \ \mathbb{P}),
\end{equation} 
where $\mathbb{P}$ is  
the underlying probability.
The first integral could be seen as a Birkhoff integral of first type. 
We can observe that the function $\Phi$ integrated with respect to the Brownian motion is a vector one and the Brownian motion could also be seen as a vector-valued function, taking values for example in $L^1\left(\Omega\right)$.
Then, the second integral in (\ref{eqbirkhoff ito}) cannot be defined as a first type and neither as a second type Birkhoff integral.
To solve this problem, we recall the Itô formula. 
If we consider as usual 
the standard Brownian Motion $w_t$ 
and the stochastic process $z_t=h(t,w_t)$, where
 the function $h:[0,T]\times \mathbb{R}\rightarrow \mathbb{R}$ is in the class 
$C^2\left( [0,T]\times \mathbb{R}\right)$, then it is:
$$d\left( z_t\right)=d\left(h(t,w_t)\right)=h'_t(t,w_t)dt+h'_x(t,w_t)dw_t+\dfrac{1}{2}h''_{xx}(t,w_t)dt$$
Obviously, in the special case that the function $h$ is linear with respect to $w_t$ and it is of type
$h(t,x)=xr(t),$
we get $h''_{xx}=0$, $h'_t=r'(t)x$ and $h'_x=r(t)$. Then we obtain
$d\left(h(t,w_t)\right)=w_tr'(t)dt+r(t)dw_t.$
This is equivalent to the following integral relation
\begin{equation}\label{ito formula}
\int_0^tr(s)dw_s=r(t)w_t-\int_0^tr'(s)w_sds.
\end{equation}
It is interesting to note that, if 
$X=\mathbb{R}$, then the equation (\ref{ito formula}) can be deduced simply applying the classic Itô formula.
If we focus on  the expression in (\ref{ito formula}), we observe that it should be used as definition of a stochastic 
integral in the Birkhoff sense.
In fact, for every fixed $t\in [0,T]$, the Brownian motion $w_t:\Omega\rightarrow \mathbb{R}$ is a scalar random
 variable and given a vector function $\Phi:[0,T]\rightarrow X$, we have that
$\Phi(t)w_t:\Omega\rightarrow X$ is an $X$-valued random variable. 
Thus, 
the first term on 
the right side of the equation (\ref{ito formula}) is an $X$-valued random variable. For the second one we need to recall the definition of \textit{
Fr\'{e}chet  
derivative,} for other definitions  see for example \cite{k0,k1,bcmed}.

\begin{defn} \rm
 A function $\Phi:[0,T] \rightarrow X$ is said to be \textit{differentiable} at a point 
$t \in [0,T]$ if exists 
 $\Phi'_t \in  \mathcal{L}\left(\mathbb{R},X\right)$, (where $\mathcal{L}\left(\mathbb{R},X\right)$ denotes the space of all continuous and linear operators from $\mathbb{R}$ to $X$)  such that
$$\lim_{h\rightarrow 0}\dfrac{\|\Phi(t+h)-\Phi(t)- \Phi'_t(h) \|_X}{|h|} = 0.$$
$\Phi'_t$ is said to be the \emph{Fr\'{e}chet 
differential of $\Phi$ at $t\in [0,T]$}. 
We say that $\Phi$ is 
\textit{differentiable on $[0,T]$} 
 if it is differentiable at every $t \in [0,T].$ 
\end{defn}

\begin{osserv} \rm 
Now, if $\Phi$ is differentiable,
we consider the $X$-valued function $t\mapsto \Phi'(t) :=\Phi'_t(1)$.
Let $(w_s)_s$ be a Brownian process. If 
 $\Phi' w$
it is first type Birkhoff integrable with respect to the Lebesgue measure, we can denote, with 
the integral notation 
${\scriptstyle (Bi_1)\hskip-.1cm}\int_0^t\Phi'(s)w_s ds$,
 an  $X$-valued random variable
defined by 
 $$\left({\scriptstyle (Bi_1)\hskip-.1cm}\int_0^t\Phi'(s)w_s ds\right) \left(\omega\right)={\scriptstyle (Bi_1)\hskip-.1cm}\int_0^t\Phi'(s)\left(w(s,\omega)\right)ds.$$
\end{osserv}
Now we are ready to define the stochastic integral of the second summand in (\ref{eqbirkhoff ito}).

\begin{defn}\label{integrale stocastico birkhoff}\rm
A function $\Phi :\mathbb{R}\rightarrow X$ 
is said to be  {\emph{stochastic $Bi_1$-integrable}} ($(Bi^{\star}_1)$ for short) with respect to a Brownian motion $w_t$ on the filtered space, \spaces \,  if
\begin{itemize}
\item $\Phi$ is differentiable on $[0,T]$ ( $\Phi'$ 
being its differential);
\item for every fixed $\omega\in \Omega$ the function $\Phi'(\cdot)w(\cdot)(\omega):[0,T]\rightarrow X$ is 
$Bi_1$-integrable with respect to the Lebesgue measure.
\end{itemize}
In this case,  for every $[a,b] \subset [0,T]$, we define 
$${\scriptstyle (Bi^{\star}_1)\hskip-.1cm}\int_a^b\Phi(r)dw_r =
\Phi(b)w_b-\Phi(a)w_a-{\scriptstyle (Bi_1)\hskip-.1cm}\int_{[a,b]}\Phi'(r)w_rdr.$$
\end{defn}
\begin{ass}\label{additional hypothesis}\rm
Suppose that $\Phi$ is strongly measurable and differentiable 
on $[0,T]$ 
and such that
for every $x^{\ast}\in X^{\ast}$, 
the function $<x^{\ast},\Phi(\cdot)> \in L^2\left(\Omega\right)$ and satisfies
\begin{equation}\label{differenziale duale}
\dfrac{d}{dt}\left(<x^{\ast},\Phi(t)>\right)=<x^{\ast},\Phi'(t)>.
\end{equation}
\end{ass}
\begin{osserv} \rm
When $X=\mathbb{R}$,
Definition \ref{integrale stocastico birkhoff}  agrees with the classic Itô formula.
In fact, given 
a  scalar function $r(t)$, it is
$\int_0^tr(s)dw_s=r(t)w_t-\int_0^tr'(s)w_sds.$
Note that in this case, the Fr\'{e}chet 
differential of $r$ is the classic differential given by
$D_t(r)(t)=r'(t)t,$ then $D_t(r)(1)=r'(t)$  and  Assumptions \ref{additional hypothesis} hold.
\end{osserv}
From now on, we suppose that $\Phi, \Psi$ are $Bi^{\star}_1$-integrable and satisfy Assumptions \ref{additional hypothesis}.
 So, we get the following
\begin{proposiz}\label{stella}
For every $x^{\ast}\in X^{\ast}$, the process $<x^{\ast},\Phi>$ is integrable with respect to the Brownian
 motion $(w_t)_t$ and moreover 
\begin{equation}\label{uguaglianza birkhoff e pettis integrale}
<x^{\ast},{\scriptstyle (Bi^{\star}_1)\hskip-.1cm}\int_0^t\Phi(s)dw_s>=\int_0^t<x^{\ast},\Phi(s)>dw_s.
\end{equation}
\end{proposiz}
\begin{proof}
Fixed $x^{\ast}\in X^{\ast}$, it is
\begin{eqnarray*}
<x^{\ast},{\scriptstyle (Bi^{\star}_1)\hskip-.1cm}\int_0^t\Phi(r)dw_r> &=& <x^{\ast},\Phi(t)w_t>-<x^{\ast},{\scriptstyle (Bi_1)\hskip-.1cm}\int_0^t \hskip-.2cm \Phi'(r)w_rdr>\\
 &=& <x^{\ast},\Phi(t)>w_t-\int_0^t \hskip-.2cm <x^{\ast},\Phi'(r)w_r>dr \\ 
&=& <x^{\ast},\Phi(t)>w_t-\int_0^t \dfrac{d}{dt}\left(<x^{\ast},\Phi(t)>\right) w_rdr\\
&=& \int_0^t \hskip-.2cm <x^{\ast},\Phi(r)>dw_r,
\end{eqnarray*}
where, in the last equalities, we have used the Assumptions \ref{additional hypothesis} and the formula  (\ref{ito formula})
applied to $h(t,w_t)=
w_t<x^{\ast},\Phi(t)>$. Then, the assertion follows.
\end{proof}

Moreover, the next result holds.
\begin{teor}\label{davedere}
The process
$(A_t)_{t\in [0,T]} :=\left({\scriptstyle (Bi^{\star}_1)\hskip-.1cm}\int_0^t\Phi(s)dw_s\right)_{t\in [0,T]}$
is a martingale with respect to the natural filtration of the Brownian motion $(w_t)_t$.
\end{teor}
\begin{proof}
Let $s\leq t$ in $[0,T]$ and fix $x^{\ast}\in X^{\ast}$. By  
Proposition \ref{stella} we have that
$<x^{\ast},\Phi(r)>$ is integrable with respect to the Brownian motion $(w_t)_t$ and
\begin{eqnarray*}
&&   <x^{\ast},\mathbb{E}\left({\scriptstyle (Bi^{\star}_1)\hskip-.1cm}  \int_0^t \Phi(r)dw_r|\mathcal{F}_s\right)> \,
 =
\mathbb{E}\left(\int_0^t<x^{\ast},\Phi(r)>dw_r|\mathcal{F}_s\right)=
\\ &&=
\mathbb{E}\left(\int_0^s<x^{\ast},\Phi(r)>dw_r\right)=
<x^{\ast},{\scriptstyle (Bi_1)\hskip-.1cm}\int_0^s\Phi(r)dw_r>.
\end{eqnarray*}
So, by arbitrariness of $x^{\ast}$, we have
$$\mathbb{E}\left({\scriptstyle (Bi^{\star}_1)\hskip-.1cm}\int_0^t\Phi(r)dw_r|\mathcal{F}_s\right)={\scriptstyle (Bi^{\star}_1)\hskip-.1cm}\int_0^s\Phi(r)dw_r$$
so the proof is completed.
\end{proof}

Now, we consider a process as follows:
$$C_t={\scriptstyle (Bi_1)\hskip-.1cm}\int_0^t\Psi(s)ds+{\scriptstyle (Bi^{\star}_1)\hskip-.1cm}\int_0^t \Phi(s) dw_s, \quad \quad (\text{under   }  \mathbb{P}),  \quad t\in [0,T]$$
(in the scalar case it is an Itô process).
We would like to eliminate the drift term so that we obtain a local martingale. To obtain it, we need a hypothesis of connection between the drift term $\Psi$ and the diffusion term $\Phi$: namely  we need to know if there exists a stochastic process $r(t)$ such that $\Psi(t) = r(t)\Phi(t)$, for every $t\in [0,T]$ (and this is the process that we use to define the change of measure).
The main peculiarity of this kind of result is that we could change the drift term, without changing the diffusion term. 

\begin{teor}{\rm ([Change of a drift term])}\label{agB}
Let $(w_s)_s$ be a Brownian motion on \spaces ,
$\Phi, \Psi$ be two processes that are $Bi^{\star}_1$-integrable and satisfy Assumptions \ref{additional hypothesis}.
 Let 
$C_t:[0,T]\times \Omega \rightarrow X$ be the  stochastic process 
defined by 
$$ C_t={\scriptstyle (Bi_1)\hskip-.1cm}\int_0^t\Psi(s)ds+{\scriptstyle (Bi^{\star}_1)\hskip-.1cm}\int_0^t \Phi(s) dw_s, \quad \quad (\text{under   }  \mathbb{P}).$$
If there exists
  $r:[0,T] \rightarrow \mathbb{R}$ be in $L^2([0,T])$ such that
  $\Psi(t) = r(t)\Phi(t)$ and 
the process    
$y_t=\exp\left\lbrace-\frac{1}{2}\int_0^t r^2(s)ds - \int_0^t r(s)dw_s\right\rbrace$
 is a martingale with respect to $\mathcal{F}$, denoted by 
 $Q^{\circ} = {\scriptstyle (Bi_1)\hskip-.1cm}\int y_T d\mathbb{P}$ and  $\tilde{w_t}=w_{t}+\int_0^t r(s)ds$,  
then  $$C_t={\scriptstyle (Bi^{\star}_1)\hskip-.1cm}\int_0^t\Phi(s)d\tilde{w}_s, \quad (\text{under} \ \ Q^{\circ}).$$
Therefore the process $C_t$ is a martingale under $Q^{\circ}$.
\end{teor}
\begin{proof}
By construction and using the classical Girsanov Theorem the process $\tilde{w_t}$ is a 
Brownian motion under the new probability  $Q^{\circ}$.
Now we claim that
\begin{eqnarray}\label{eq pettis}
{\scriptstyle (Bi^{\star}_1)}\int_0^t\Phi(s)d\tilde{w}_s=
{\scriptstyle (Bi^{\star}_1)}\int_0^t\Phi(s)dw_s +
{\scriptstyle (Bi_1)}\int_0^t r(s)\Phi(s)ds
\end{eqnarray}
as a vector equivalence between the
 Birkhoff stochastic integral and the Birkhoff integral.
To prove this we consider, for every $x^{\ast}\in X^{\ast}$,
\begin{eqnarray*}
&& <x^{\ast},
{\scriptstyle (Bi^{\star}_1)}\int_0^t\Phi(s)d\tilde{w}_s>
= \int_0^t<x^{\ast},\Phi(s)>d\tilde{w}_s=
\\ &=&\int_0^t<x^{\ast},\Phi(s)>\left(dw_s + r(s)ds\right)=\\
&=&\int_0^t<x^{\ast},\Phi(s)>dw_s + \int_0^t<x^{\ast},\Phi(s)>r(s)ds=\\
&=&<x^{\ast}, {\scriptstyle (Bi^{\star}_1)}\int_0^t \Phi(s)dw_s> +
<x^{\ast},{\scriptstyle (Bi^{\star}_1)}\int_0^t \Phi(s)r(s)ds>=\\
&=&<x^{\ast},{\scriptstyle (Bi^{\star}_1)}\int_0^t \Phi(s)dw_s +
{\scriptstyle (Bi_1)}\int_0^t\Phi(s)r(s)ds>.
\end{eqnarray*}
So, by the arbitrariness of $x^{\ast}\in X^{\ast}$, 
the equation (\ref{eq pettis}) holds. Thus,
\begin{eqnarray*}
C_t&=&{\scriptstyle (Bi^{\star}_1)}\int_0^t\Psi(s)ds+
{\scriptstyle (Bi^{\star}_1)}\int_0^t\Phi(s)dw_s=\\
&=&{\scriptstyle (Bi^{\star}_1)}\int_0^t \Psi(s)ds+
{\scriptstyle (Bi^{\star}_1)}\int_0^t\Phi(s)d\tilde{w}_s-
{\scriptstyle (Bi^{\star}_1)}\int_0^tr(s)\Phi(s)ds= \\ &=&
{\scriptstyle (Bi^{\star}_1)}\int_0^t\Phi(s)d\tilde{w}_s,
\end{eqnarray*}
and then we have that, 
$$C_t={\scriptstyle (Bi^{\star}_1)}\int_0^t\Psi(s)d\tilde{w}_s \quad (\mbox{under } Q^{\circ})$$ and it turns out to be a 
martingale, thanks to Theorem \ref{davedere}.
\end{proof}

\section*{Acknowledgements}
This is a post-peer-review, pre-copyedit version of an article published in Mediterranean Journal of  Mathematics. The final authenticated version is available online at: http://dx.doi.org/10.1007\%2Fs00009-019-1431-x.
\small
\par\noindent\rule{\textwidth}{1pt}\\

\noindent\begin{minipage}{0.25\textwidth}
\includegraphics[width=0.98\textwidth]{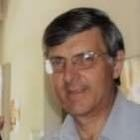} 
\end{minipage}
\hfill
\begin{minipage}{0.75\textwidth} \footnotesize
Udine $\star$ 10/18/1951, Rome \dag \, 05/03/2019.
His research was devoted to Measure Theory and Integration, with particular concern for:
    many different types of integrals, with applications to Calculus of Variations, Stochastic Integration,  infinite-dimensional integrals; multivalued integration with applications  to
  the  properties of measurable
and integrable multifunctions and to the search of Equilibria; finitely additive measures, concerning structural problems (non-atomicity, extensions and restrictions, decompositions) and applications to Probability, in particular Stochastic integration for infinite-dimensional Processes.
\end{minipage}

\end{document}